\documentclass[11pt,a4paper,fleqn,final]{article}

\usepackage{preprint-modern}

\crefname{appsec}{Appendix}{Appendices}
\Crefname{appsec}{Appendix}{Appendices}

\usepackage[
backend=biber,
bibencoding=utf8,
giveninits=true,
style=alphabetic,
maxnames=4,
isbn=false,
doi=true,
eprint=false,
url=false,
date=year
]{biblatex}
\AtEveryBibitem{%
  \ifentrytype{online}{}{%
    \clearfield{url}%
    \clearfield{urldate}%
  }%
}

\AtEveryBibitem{\clearfield{pubstate}}
\DeclareFieldFormat{doi}{%
  \href{https://doi.org/#1}{doi.org/#1}}

\DeclareFieldFormat{eprint:arxiv}{%
   \href{https://arxiv.org/abs/#1}{arXiv:#1}}

\DeclareFieldFormat{eprint:hal}{%
    \href{https://hal.science/#1}{#1}}

\renewbibmacro*{doi+eprint+url}{%
  \ifentrytype{online}{%
    \iffieldundef{eprinttype}{%
      \usebibmacro{url+urldate}%
    }{%
      \usebibmacro{eprint}%
    }%
  }{%
    \iftoggle{bbx:eprint}{%
      \iffieldundef{eprint}{%
        \iftoggle{bbx:doi}{\printfield{doi}}{}%
      }{%
        \usebibmacro{eprint}%
      }%
    }{%
      \iftoggle{bbx:doi}{\printfield{doi}}{}%
    }%
  }%
}

\begin{document}
	
	\title{A dispersive extension result for a class of nonlocal nonlinearities}

	\author[0000-0002-0329-9402]{Bastian Hilder}{Department of Mathematics, Technische Universität München, Boltzmannstraße 3, 85748 Garching b.\ München, Germany}{bastian.hilder@tum.de}
	\author[0000-0002-7063-6173]{Christian Kuehn}{Department of Mathematics, Technische Universität München, Boltzmannstraße 3, 85748 Garching b.\ München, Germany}{ckuehn@ma.tum.de}
	
	
	\maketitle

\begin{abstract}
    We prove that a class of nonlocal nonlinearities on periodic Sobolev spaces can be expressed as a trace of the unique mild solution to a local dispersive problem consisting of a system of forced Schrödinger equations. The proof is based on a reformulation of the nonlocal nonlinearities as the resonant orbit average of a real-analytic function under the unitary group generated by the free Schrödinger operator $\ii \partial_x^2$. We illustrate our result by providing an equivalent local reformulation for a recently obtained nonlocal amplitude equation that formally captures the dynamics of parabolic systems close to a conserved Hopf instability.
\end{abstract}


\section{Introduction}

Nonlocal partial differential equations appear frequently across a wide range of applications in mechanics \cite{nagel2025-03SIAMRev} and life sciences \cite{pal2025-07PhysLifeRev}. 
A specific recent example is the nonlocal amplitude equation
\begin{equation}\label{eq:amplitude-equation}
  \partial_t a = \partial_x^2 \Big[-\mu a - \partial_x^2 a + \frac{3}{2}(2 a \langle|a|^2\rangle - \Fcal^{-1}[\hat{a} |\hat{a}|^2])\Big]
\end{equation}
with $x \in (0,L)$ for some $L > 0$ and $t \geq 0$, formally derived from a system of non-reciprocally coupled Cahn–Hilliard equations close to a conserved Hopf instability; see \cite{greve2024-12Chaos}. Here $a = a(t,x) \in \C$ is the complex amplitude of a fast oscillatory mode, $\mu > 0$ is the distance to the onset of instability, and periodic boundary conditions are assumed. Additionally, $\langle \cdot \rangle$ denotes the spatial average and $\Fcal^{-1}$ the inverse discrete Fourier transform. It is conjectured in \cite{greve2024-12Chaos} that equation \eqref{eq:amplitude-equation} appears universally in partial differential equations (PDEs) close to a conserved Hopf instability, which occurs, for example, in pattern formation of MIN proteins and active matter systems; see \cite{frohoff-hulsmann2023-09PhysRevLett} for further references. It is well known that nonlocal averaging terms can also occur in amplitude equations close to a Turing–Hopf instability; see, for example, \cite{schneider1997-01ProcRoySocEdinburghSectA} and \cite[Remark 10.7.6]{schneider2017book}. Note carefully that \eqref{eq:amplitude-equation} contains two nonlocal terms that are also acted upon by a differential operator. This is substantially different from semilinear nonlocal PDEs~\cite{morgan2014-03PhysD}, which often even reduce to local amplitude equations~\cite{kuehn2018-07JMathPhys,kuehn2025-02PhysDa}.

Naturally, the tools to analyse nonlocal PDEs are less developed than their local counterparts. Therefore, a common strategy is to find an equivalent formulation of the problem, where only local terms appear. A canonical example is the reformulation of integral convolution terms whose kernel is the Green's function of a local elliptic operator. In this case, it is possible to obtain an equivalent local system by extending the system with an additional local elliptic equation to replace the convolution integral. This equivalence has for example been used in \cite{gajewski2004-03NonlinearAnal,porta2015DiscreteContinDynSystSerB}.

A second canonical example is the localisation of fractional powers of elliptic operators by introducing an additional spatial coordinate. For the fractional Laplacian $(-\Delta)^s$ with $s \in (0,1)$, Caffarelli and Silvestre \cite{caffarelli2007-08CommPartialDifferentialEquations} proved that $(-\Delta)^s$ can be obtained from an elliptic extension to the upper half-plane that maps the boundary Dirichlet data to the boundary Neumann data. The corresponding nonlocal operators are called Dirichlet–to–Neumann operators and appear frequently, for example, in reformulations of the water-wave problem when it is reduced to an equation on the fluid surface; see \cite{lannes2013-05book,groves2020-02ProcRSocA}. For fractional powers of more general operators, such extension results have been proved in \cite{stinga2010-10CommPartialDifferentialEquations,gale2013-06JEvolEqu}; see also \cite{biswas2024-08JFunctAnal} for an extension result for general powers $s > 0$. In addition, \cite{kemppainen2015DiscreteContinDynSyst} provides a different extension of the fractional Laplacian via a hyperbolic problem.

To the best of our knowledge, no such localisation result exists for the nonlocal nonlinearities appearing in the amplitude equation \eqref{eq:amplitude-equation}. In this paper, we prove that a class of resonant nonlinearities (see \cref{def:nonlinearity}), which includes the nonlinearity in \eqref{eq:amplitude-equation}, admits a local formulation as a system of forced Schrödinger equations; see \cref{thm:main-result}. More precisely, we introduce the auxiliary system
\begin{equation}\label{eq:schrödinger-eq-intro}
  \begin{split}
    \partial_\tau v &= \ii \partial_x^2 v, \\
    \partial_\tau w &= \ii \partial_x^2 w + \rho(v)
  \end{split}
\end{equation}
for some function $\rho$ depending on the nonlinearity. We then show that the chosen nonlinearity at a periodic function $u$ can be recovered from the mild solution to \eqref{eq:schrödinger-eq-intro} with initial data $(v,w)(0) = (u,0)$ at $\tau = \tau_P$, where $\tau_P > 0$ is the period of the periodic free Schrödinger group $e^{i \tau \partial_x^2}$. Applying this result to the amplitude equation \eqref{eq:amplitude-equation} then yields an equivalent local reformulation of the amplitude equation close to a conserved Hopf instability; see \cref{cor:localisation-amplitude-eq}. Note that, while we formulate the theory in terms of the free Schrödinger operator, \cref{thm:main-result} can be extended directly to more general dispersive operators; see \cref{sec:generalisations}.

This construction is conceptually similar to the elliptic extension by Caffarelli and Silvestre \cite{caffarelli2007-08CommPartialDifferentialEquations} in that a nonlocal term is recovered from the trace data of an extended local problem. However, the underlying mechanism is very different. Specifically, we introduce an additional \emph{temporal} direction instead of a \emph{spatial} direction and the extended problem is dispersive rather than elliptic.

\paragraph{Outline}
The paper is organised as follows. In \cref{sec:localisation-result}, we define the relevant class of resonant nonlocal nonlinearities, and prove that these can be expressed as evaluations of mild solutions to a system of local forced Schrödinger equations; see \cref{thm:main-result}. In \cref{sec:applications}, we then apply the result to obtain a local reformulation of a class of nonlocal evolution equations (see \cref{prop:equivalence}) and specifically give a local reformulation of the amplitude equation \eqref{eq:amplitude-equation}; see \cref{cor:localisation-amplitude-eq}. Finally, we discuss possible extensions of the class of relevant nonlinearities in \cref{sec:generalisations}.

\section{The localisation result}\label{sec:localisation-result}

Before stating the results, we introduce some notation used throughout the paper. Fix a spatial domain $(0,L)$ of length $L > 0$. For any $\ell > 1/2$ we denote by $H^\ell_\per(0,L)$ the Sobolev space of complex-valued, periodic functions on $(0,L)$. Due to the assumption $\ell > 1/2$, $H^\ell_\per(0,L)$ is a multiplication algebra with the norm
\begin{equation*}
  \|u\|_{H^\ell_\per} = C \Big(\sum_{K \in \Z} (1+ |K|^2)^\ell |\hat{u}_K|^2\Big)^{\frac{1}{2}}, \qquad u(x) = \sum_{K \in \Z} \hat{u}_{K} \ee^{\ii \frac{2\pi K x}{L}},
\end{equation*}
where the constant $C$ is chosen such that the norm is submultiplicative. Since $\ell > 1/2$, we have that $H^\ell_\per(0,L)$ embeds into $L^\infty(0,L)$, and, by choosing a potentially larger $C$, we can guarantee that $\|u\|_{L^\infty(0,L)} \leq \|u\|_{H^\ell_\per(0,L)}$. Specifically, this yields that $\sum_{K \in \Z} |\hat{u}_K| < \infty$. Finally, we denote the Fourier transform on $H^\ell_\per(0,L)$ and its inverse by $\Fcal$ and $\Fcal^{-1}$, respectively.

For $a \in H^\ell_\per(0,L)$ we can then define the spatial average
\begin{equation*}
  \langle |a|^2 \rangle \coloneqq \sum_{K \in \Z} |\hat{a}_K|^2.
\end{equation*}
For notational convenience, we also denote the constant-one vector in $\Z^a$ for $a \in \N$ by $\oneb_a = (1,\dots, 1)$.

As discussed in the introduction, the localisation result is based on a dispersive extension using the free Schrödinger operator $\ii \partial_x^2$. Therefore, we recall some basic properties here. The free Schrödinger operator generates a unitary group $S_{\tau} \coloneqq \ee^{\ii \tau \partial_x^2}$ on $H^\ell_\per(0,L)$ and $S_{\tau} v_0 \in C^0(\R;H^\ell_\per(0,L))$ is the unique mild solution of the free Schrödinger equation in $H^\ell_\per(0,L)$ with initial value $v_0$ at $\tau = 0$. The action of $S_{\tau}$ on $u \in H^\ell_\per(0,L)$ can be expressed using the Fourier series
\begin{equation}\label{eq:representation-free-schrödinger-group}
  S_{\tau} u = \sum_{K \in \Z} \ee^{-\ii \tau \big(\frac{2\pi K}{L}\big)^2} \hat{u}_K \ee^{\ii \frac{2\pi K x}{L}}.
\end{equation}
The group $S_{\tau}$ is unitary on $H^\ell_\per(0,L)$, that is, $\|S_{\tau} v\|_{H^\ell_\per} = \|v\|_{H^\ell_\per}$. Finally, due to the finite length of the interval $(0,L)$, the flow $S_\tau$ is periodic with period $\tau_P \coloneqq \frac{L^2}{2\pi}$, which follows directly from the representation \eqref{eq:representation-free-schrödinger-group}.

We now define the class of nonlocal nonlinearities for which the main localisation result \cref{thm:main-result} is valid.

\begin{definition}\label{def:nonlinearity}
  Fix $\ell > 1/2$ and let $\rho : \C \to \C$ be real analytic such that
  \begin{equation}\label{eq:series-convergence}
    \rho(z) = \sum_{\substack{a,b \geq 0 \\ a+b \geq 2}} \alpha_{a,b} z^a \bar{z}^b, \qquad \text{with} \sum_{\substack{a,b \geq 0 \\ a+b \geq 2}} |\alpha_{a,b}| r^{a+b} < \infty
  \end{equation}
  for all $r < R \leq \infty$. For $u \in H^\ell_\per(0,L)$ with $\|u\|_{H^\ell_\per} < R$, define the corresponding nonlinearity $\Ncal_\rho$ via its Fourier symbol by
  \begin{equation}\label{eq:resonant-nonlinearity}
    \widehat{\Ncal_\rho(u)}_K = \sum_{\substack{a,b \geq 0 \\ a+b \geq 2}} \alpha_{a,b} \Big[ \sum_{(\j,\k) \in \Rcal_{a,b}(K)} \prod_{m=1}^a \hat{u}_{\j_m} \prod_{m=1}^b \bar{\hat{u}}_{\k_m} \Big]
  \end{equation}
  where for each wave number $K \in \Z$, the set of resonant Fourier modes $\Rcal_{a,b}(K)$ is defined as
  \begin{equation*}
    \Rcal_{a,b}(K) \coloneqq \{(\j,\k) \in \Z^a \times \Z^b \,:\, K^2 = |\j|^2 - |\k|^2 \text{ and } K = \oneb_a \cdot \j - \oneb_b \cdot \k \}.
  \end{equation*}
\end{definition}

\begin{remark}
  In \eqref{eq:resonant-nonlinearity} and throughout the paper, we use the conventions $|\j|^2 = 0$ and $\oneb_0 \cdot \j = 0$ for all $\j \in \Z^0$. Additionally, we set $\prod_m^0 (\cdot) = 1$.
\end{remark}

\cref{def:nonlinearity} in particular includes the nonlinearity in the nonlocal amplitude equation \eqref{eq:amplitude-equation}. Indeed, with the choice $\rho(u) = u^2 \bar{u} = u |u|^2$, we find
\begin{equation*}
  \widehat{\Ncal_\rho(u)}_K = 2\hat{u}_K \sum_{Q \in \Z} |\hat{u}_Q|^2 - \hat{u}_K |\hat{u}_K|^2
\end{equation*}
by observing that for $a = 2$ and $b = 1$, the resonance conditions in $\Rcal_{2,1}(K)$ yield
\begin{equation*}
  0 = K_1^2 + K_2^2 - K_3^2 - K^2 \text{ and } 0 = K_1 + K_2 - K_3 - K
\end{equation*}
with $\j = (K_1,K_2)$ and $\k = (K_3)$; see \cite[Equations (14)--(16)]{greve2024-12Chaos} for detailed calculations.

We briefly comment on the set of resonant Fourier modes $\Rcal_{a,b}(K)$. The second condition $K = \oneb_a \cdot \j - \oneb_b \cdot \k$ is the standard condition for a convolution of Fourier modes at Fourier wave number $K \in \Z$. We specifically note that the indices $\k$ act on the complex conjugate of $u$ and therefore enter the balance with a minus sign. As we will see later, the first condition $K^2 = |\j|^2 - |\k|^2$ identifies the nonlinear interactions of Fourier modes which are resonant under a Schrödinger orbit average; see also \cref{rem:resonance-condition}.

Our main result now states that the nonlocal nonlinear terms in \cref{def:nonlinearity} can be expressed as a trace of the mild solution to a system of forced Schrödinger equations, where the forcing term is precisely determined by the generating function $\rho$.

\begin{theorem}\label{thm:main-result}
  Let $\ell > \tfrac{1}{2}$, and let $\Ncal_\rho$ be a nonlinearity satisfying \cref{def:nonlinearity} with convergence radius $R$. Then, for all $u \in H^\ell_\per(0,L)$ with $\|u\|_{H^\ell_\per} < R$, it holds that
  \begin{equation*}
    \begin{split}
      \Ncal_\rho(u) = \tau_P^{-1} w(\tau_P)
    \end{split}
  \end{equation*}
  where $(v,w) \in C^0([0,\tau_P];H^\ell_\per(0,L) \times H^\ell_\per(0,L))$ is the unique mild solution to
  \begin{equation}\label{eq:forced-Schrödinger-system}
    \begin{split}
        \partial_\tau v &= \ii \partial_x^2 v, \\
        \partial_\tau w &= \ii \partial_x^2 w + \rho(v)
    \end{split}
  \end{equation}
  with initial data $(v,w)(0) = (u,0) \in H^\ell_\per(0,L) \times H^\ell_\per(0,L)$.
\end{theorem}

The key observation in the proof is that the nonlocal nonlinearities in \cref{def:nonlinearity} can be expressed as an orbit average of the unitary group $S_\tau$ generated by the free Schrödinger operator $\ii\partial_x^2$. Similar integral nonlinearities occur in the analysis of dispersion-managed nonlinear Schrödinger (NLS) equations; see \cite{lushnikov2001-10OptLett} for a numerical example and \cite{choi2023-06JMathAnalAppl} for an analytical one. Therefore, such a result is also of independent interest.

\begin{lemma}\label{lem:integral-representation}
  Let $\ell > 1/2$ and let $\Ncal_\rho$ be a nonlinearity satisfying \cref{def:nonlinearity} with convergence radius $R$. Then, $\Ncal_\rho$ is a well-defined map from $\Bcal_R$ into $H^\ell_\per(0,L)$, where $\Bcal_R$ denotes the open ball in $H^\ell_\per(0,L)$ with radius $R$, and for all $u \in \Bcal_R$ it holds that
  \begin{equation}\label{eq:integral-representation}
    \Ncal_\rho(u) = \tau_P^{-1} \int_0^{\tau_P} S_{-\tau} (\rho(S_\tau u)) \dd\tau.
  \end{equation}
\end{lemma}

\begin{proof}
  Let $u \in H^\ell_\per(0,L)$ with $\|u\|_{H^\ell_\per} < R$. Since $\rho : \C \to \C$ is real analytic and satisfies \eqref{eq:series-convergence}, the map $u \mapsto \rho(u)$ is well-defined and real-analytic from $\Bcal_R$ into $H^\ell_\per(0,L)$. Indeed, the series representation satisfies
  \begin{equation*}
    \sum_{\substack{a,b \geq 0 \\ a+b \geq 2}} |\alpha_{a,b}| \|u\|_{H^\ell_\per(0,L)}^{a+b} \leq \sum_{\substack{a,b \geq 0 \\ a+b \geq 2}} |\alpha_{a,b}| r^{a+b} < \infty
  \end{equation*} 
  for some $r$ satisfying $\|u\|_{H^\ell_\per(0,L)} \leq r < R$, and therefore converges in $H^\ell_\per(0,L)$.
  Here, we use that any function $u \in \Bcal_R$ satisfies $\|u\|_{L^\infty} < R$. Next, since $S_\tau$ is a unitary group on $H^\ell_\per(0,L)$, it maps $\Bcal_R$ into itself. Therefore, we have that
  \begin{equation*}
    \int_0^{\tau_P} S_{-\tau}(\rho(S_\tau u)) \dd \tau = \sum_{\substack{a,b \geq 0 \\ a+b \geq 2}} \alpha_{a,b} \int_0^{\tau_P} S_{-\tau}((S_\tau u)^a \overline{S_{\tau} u}^b) \dd\tau.
  \end{equation*}
  Interchanging the sum and the integral is permitted due to the dominated convergence theorem for Bochner integrals using again that $S_\tau$ is unitary and that the power series is absolutely convergent on a disk of radius $R$. Therefore, it is sufficient to prove \eqref{eq:integral-representation} for a monomial $\rho(z,\bar{z}) = z^a \bar{z}^b$ with $a,b \geq 0$ and $a+b \geq 2$. Recalling \eqref{eq:representation-free-schrödinger-group}, we first calculate
  \begin{equation*}
    \begin{split}
      (S_\tau u)^a \overline{(S_\tau u)}^b &= \Big(\sum_{K_1 \in \Z} \ee^{-\ii \tau \big(\frac{2\pi K_1}{L}\big)^2} \hat{u}_{K_1} \ee^{\ii \frac{2\pi K_1 x}{L}}\Big)^a \Big(\sum_{K_2 \in \Z} \ee^{\ii \tau \big(\frac{2\pi K_2}{L}\big)^2} \bar{\hat{u}}_{K_2} \ee^{-\ii \frac{2\pi K_2 x}{L}}\Big)^b \\
      &= \sum_{\j \in \Z^a, \k \in \Z^b} \ee^{-\ii \tau \big(\frac{2\pi}{L}\big)^2 (|\j|^2 - |\k|^2)} \ee^{\frac{2\pi \ii x}{L} (\oneb_a \cdot \j - \oneb_b \cdot \k)} \prod_{m=1}^a \hat{u}_{\j_m} \prod_{m=1}^b \bar{\hat{u}}_{\k_m}.
    \end{split}
  \end{equation*}
  This equality follows by noting that every series is absolutely convergent since $u \in H^\ell_\per$ with $\ell > 1/2$.
  The Fourier mode with wave number $Q \in \Z$ of $(S_\tau u)^a \overline{(S_\tau u)}^b$ is therefore given by
  \begin{equation*}
    \Fcal[(S_\tau u)^a \overline{(S_\tau u)}^b]_Q = \sum_{\substack{\j \in \Z^a, \k \in \Z^b \\ Q = \oneb_a \cdot \j - \oneb_b \cdot \k}} \ee^{-\ii \tau \big(\frac{2\pi}{L}\big)^2 (|\j|^2 - |\k|^2)} \prod_{m=1}^a \hat{u}_{\j_m} \prod_{m=1}^b \bar{\hat{u}}_{\k_m}
  \end{equation*}
  Next, we calculate
  \begin{equation*}
    \begin{split}
      \int_0^{\tau_P} S_{-\tau} ((S_\tau u)^a \overline{(S_\tau u)}^b) \dd \tau &= \int_0^{\tau_P} \sum_{Q \in \Z} \sum_{\substack{\j \in \Z^a, \k \in \Z^b \\ Q = \oneb_a \cdot \j - \oneb_b \cdot \k}} \ee^{-\ii \tau \big(\frac{2\pi}{L}\big)^2 (|\j|^2 - |\k|^2 - Q^2)} \ee^{\frac{2\pi \ii Q x}{L}} \prod_{m=1}^a \hat{u}_{\j_m} \prod_{m=1}^b \bar{\hat{u}}_{\k_m} \\
      &= \sum_{Q \in \Z} \sum_{\substack{\j \in \Z^a, \k \in \Z^b \\ Q = \oneb_a \cdot \j - \oneb_b \cdot \k}} \ee^{\frac{2\pi \ii Q x}{L}} \prod_{m=1}^a \hat{u}_{\j_m} \prod_{m=1}^b \bar{\hat{u}}_{\k_m} \int_0^{\tau_P} \ee^{-\ii \tau \big(\frac{2\pi}{L}\big)^2 (|\j|^2 - |\k|^2 - Q^2)} \dd \tau
    \end{split}
  \end{equation*}
  where we used in the second equality that $(\hat{u}_K)_{K \in \Z}$ is absolutely summable since $u \in H^\ell_\per(0,L)$ with $\ell > 1/2$. This yields that the sum is bounded independent of $\tau$ and $x$ and therefore, the sum and integral can be interchanged by Fubini's theorem. The integral vanishes if $Q^2 \neq |\j|^2 - |\k|^2$ and equals $\tau_P$ otherwise. Recalling the definition of the set of resonant Fourier modes $\Rcal_{a,b}(Q)$ in \cref{def:nonlinearity} we therefore obtain
  \begin{equation*}
    \int_0^{\tau_P} S_{-\tau} ((S_\tau u)^a \overline{(S_\tau u)}^b) \dd \tau = \tau_P \sum_{Q \in \Z} \sum_{(\j,\k) \in \Rcal_{a,b}(Q)} \ee^{\frac{2\pi \ii Q x}{L}} \prod_{m=1}^a \hat{u}_{\j_m} \prod_{m=1}^b \bar{\hat{u}}_{\k_m},
  \end{equation*}
  which completes the proof.
\end{proof}

\begin{remark}\label{rem:resonance-condition}
  From the proof of \cref{lem:integral-representation} we find that the condition $Q^2 = |\j|^2 - |\k|^2$ identifies the nonlinear interactions of Fourier modes which are resonant under averaging with respect to the periodic Schrödinger group $S_{\tau}$. Such resonant conditions are standard in resonant averaging (see, e.g., \cite{kuksin2018-04ProcRoySocEdinburghSectA}) and resonant wave interactions (see, e.g., \cite{zakharov1992book}). Recalling that each Fourier mode of $v$ satisfies $\partial_\tau \hat{v}_K = - \ii (\frac{2\pi K}{L})^2 \hat{v}_K$ for $K \in \Z$, this can also be interpreted as the resonant nonlinear terms in a normal-form transformation. Here, $Q^2 = |\j|^2 - |\k|^2$ is exactly the condition that the corresponding eigenvalues of the Fourier modes are resonant; see, for example, \cite{craig2013-06AttiAccadNazLinceiClSciFisMatNatur}.
\end{remark}

To complete the proof of \cref{thm:main-result} it remains to prove that the forced Schrödinger system \eqref{eq:forced-Schrödinger-system} admits a unique mild solution and that the integral in \eqref{eq:integral-representation} can be expressed as a trace of this solution.

\begin{proof}[Proof of \cref{thm:main-result}]
  We first show that the forced Schrödinger system
  \begin{equation}\label{eq:schrödinger-system-proof}
    \begin{split}
      \partial_\tau v &= \ii \partial_x^2 v, \\
      \partial_\tau w &= \ii \partial_x^2 w + \rho(v)
    \end{split}
  \end{equation}
  with initial conditions $(v,w)(0) = (u,0) \in \Bcal_R \times H^\ell_\per(0,L)$ has a unique mild solution $(v,w) \in C^0([0,\tau_P];\Bcal_R \times H^\ell_\per(0,L))$. We recall that the free Schrödinger operator $\ii \partial_x^2$ generates the group $S_{\tau}$ and therefore, $v(\tau) = S_{\tau} u \in C^0([0,\tau_P];H^\ell_\per(0,L))$. Since $S_\tau$ is unitary and $\|u\|_{H^\ell_\per} < R$, we find that $\|v(\tau)\|_{H^\ell_\per} < R$ for all $\tau \in \R$. As in the proof of \cref{lem:integral-representation}, we find that $u \mapsto \rho(u)$ is a real-analytic map from $\Bcal_R$ into $H^\ell_\per(0,L)$. Since $\tau \mapsto v(\tau)$ is continuous, $\tau \mapsto \rho(v(\tau))$ is continuous as a composition of continuous functions. Therefore, the $w$-equation in \eqref{eq:schrödinger-system-proof} has a unique mild solution given by the Duhamel formulation
  \begin{equation*}
    w(\tau) = \int_0^\tau S_{\tau-s} \rho(S_{s}u) \dd s,
  \end{equation*}
  see, for example, \cite{pazy1983book}. Finally, we use that $\tau \mapsto S_\tau$ is periodic with period $\tau_P$. Therefore, we obtain $S_{\tau_P-s} = S_{-s}$ and thus
  \begin{equation*}
    w(\tau_P) = \int_0^{\tau_P} S_{-s} (\rho(S_s u)) \dd s.
  \end{equation*}
  Together with \cref{lem:integral-representation}, this completes the proof of \cref{thm:main-result}.
\end{proof}

\section{Application to evolution equations}\label{sec:applications}

We now illustrate how the local reformulation result \cref{thm:main-result} applies to nonlocal evolution problems. We first state a general equivalence result for solutions of a nonlocal evolution equation and its extended local counterpart. Then, we apply the theory to the amplitude equation \eqref{eq:amplitude-equation}.

\begin{proposition}\label{prop:equivalence}
  Fix $\ell > 1/2$ and $n \in \N_0$, and let $\Lcal : \Dcal(\Lcal) \to H^\ell_\per(0,L)$ be the generator of a strongly continuous semigroup $e^{t\Lcal}$. Assume that $e^{t\Lcal}\partial_x^n$, defined on the dense domain $H^{\ell+n}_\per(0,L)$, extends to a bounded linear operator $B_n(t)$ on $H^\ell_\per(0,L)$ for all $t \in (0,T)$ and satisfies
  \begin{equation}\label{eq:smoothing-property}
    \|B_n(t) u\|_{H^\ell_\per(0,L)} \leq C t^{-\gamma} \|u\|_{H^\ell_\per(0,L)}
  \end{equation} 
  for all $t \in (0,T)$ with $\gamma \in [0,1)$.
  Let $\Ncal_\rho$ be a nonlinearity satisfying \cref{def:nonlinearity}. Then, $u \in C^0([0,T];\Bcal_R)$ is a mild solution to
  \begin{equation}\label{eq:nonlocal-evolution-equation}
    \partial_t u = \Lcal u + \partial_x^n \Ncal_\rho(u), \quad \text{i.e.}\quad u(t) = e^{t \Lcal} u(0) + \int_0^t B_n(t-s) \Ncal_\rho(u(s)) \dd s,
  \end{equation}
  if and only if the triple $(u,v,w)$ with $u \in C^0([0,T];\Bcal_R)$ and $(v,w) \in C^0([0,T] \times [0,\tau_P];\Bcal_R \times H^\ell_\per(0,L))$ is a mild solution to the extended local problem
  \begin{equation}\label{eq:local-extended-equation}
    \begin{split}
      \partial_t u &= \Lcal u + \tau_P^{-1} \partial_x^n w(\cdot, \tau_P), \\
      \partial_\tau v &= \ii  \partial_x^2 v, \\
      \partial_\tau w &= \ii \partial_x^2 w + \rho(v).
    \end{split}
  \end{equation}
  with $(v,w)(t,0) = (u(t),0)$.
\end{proposition}
\begin{proof}
  We first prove that solutions to \eqref{eq:nonlocal-evolution-equation} yield solutions to \eqref{eq:local-extended-equation}. Assume that $u \in C^0([0,T];\Bcal_R)$ is a mild solution to \eqref{eq:nonlocal-evolution-equation}.
  For every fixed $t \in [0,T]$, \cref{thm:main-result} yields that there exists a mild solution $(v,w)(t,\cdot) \in C^0([0,\tau_P];H^\ell_\per(0,L))$ to the forced Schrödinger system $\eqref{eq:local-extended-equation}_2$--$\eqref{eq:local-extended-equation}_3$. Additionally, for every $t \in [0,T]$ we then have $\Ncal_\rho(u(t)) = \tau_P^{-1} w(t, \tau_P)$. In fact, the proof of \cref{thm:main-result} provides an explicit representation formula, which in particular shows continuous dependence on the initial data $u$ in $H^\ell_\per(0,L)$. Finally, joint continuity in $(t,\tau)$ follows from the representation of $(v,w)$ together with joint continuity of $(\tau,v) \mapsto S_{\tau} v$ and continuity of $v \mapsto \rho(v)$. This completes this part of the proof.

  The other direction follows immediately from \cref{thm:main-result}.
\end{proof}

\begin{remark}
  If $n = 0$, the assumptions in \cref{prop:equivalence} are satisfied for all strongly continuous semigroups. Therefore, the result specifically applies to nonlocal dispersive equations, as they appear for example in the context of dispersion-managed NLS equations. If $n > 0$, the assumptions require smoothing properties of the semigroup. These are satisfied, for example, if $\Lcal$ is a strictly elliptic operator of order $m > n$.
\end{remark}

\begin{remark}
  \cref{prop:equivalence} provides equivalence of mild solutions between the nonlocal equation \eqref{eq:nonlocal-evolution-equation} and the corresponding local extension \eqref{eq:local-extended-equation}. The same result also holds for classical solutions $u \in C^1([0,T];H^\ell_\per(0,L)) \cap C^0([0,T];\Dcal(\Lcal))$ in the sense that $(v,w)$ are still mild solutions to the forced Schrödinger equations $\eqref{eq:local-extended-equation}_2$--$\eqref{eq:local-extended-equation}_3$ which are $C^1$ in $t$. That is, $(v,w) \in C^1([0,T];C^0([0,\tau_P];H^\ell_\per(0,L)))$. Here, $C^1$ in $t$ is obtained from differentiating the explicit solution formulas for $(v,w)$ with respect to $t$.
\end{remark}

We now apply \cref{prop:equivalence} to the nonlocal amplitude equation \eqref{eq:amplitude-equation}. As pointed out below \cref{def:nonlinearity}, the nonlinearity $2 a \langle |a|^2\rangle - \Fcal^{-1}(\hat{a} |\hat{a}|^2)$ indeed satisfies \cref{def:nonlinearity} with generating function $\rho(v) = v^2 \bar{v} = v |v|^2$. The linear operator is given by
\begin{equation*}
  \Lcal = - \partial_x^4 - \mu \partial_x^2.
\end{equation*}
This operator generates an analytic semigroup on $H^\ell_\per(0,L)$ with domain $H^{\ell+4}_\per(0,L)$. Therefore, the extension property and smoothing estimate \eqref{eq:smoothing-property} for $e^{-t(\partial_x^4+\mu\partial_x^2)} \partial_x^2$ are automatic; see, for example, \cite[Chapter 2, Theorem 6.13]{pazy1983book}. An application of \cref{prop:equivalence} then yields the following result.

\begin{corollary}\label{cor:localisation-amplitude-eq}
  Let $\ell > 1/2$. Then, $a \in C^0([0,T];H^\ell_\per(0,L))$ is a mild solution to \eqref{eq:amplitude-equation} if and only if the triple $(a,v,w)$ with $a \in C^0([0,T];H^\ell_\per(0,L))$ and $(v,w) \in C^0([0,T] \times [0,\tau_P];H^\ell_\per(0,L) \times H^\ell_\per(0,L))$ is a mild solution to the extended local problem
  \begin{equation}\label{eq:local-extended-amplitude equation}
    \begin{split}
      \partial_t a &= \partial_x^2\Big[-\mu a - \partial_x^2 a + \frac{3}{2\tau_P} w(t,\tau_P)\Big], \\
      \partial_\tau v &= \ii  \partial_x^2 v, \\
      \partial_\tau w &= \ii \partial_x^2 w + v |v|^2
    \end{split}
  \end{equation}
  with $(v,w)(t,0) = (a(t),0)$.
\end{corollary}

\section{Generalisations}\label{sec:generalisations}

We have so far restricted the results to nonlinearities that can be localised via an extension with a forced Schrödinger system. We conclude the paper by discussing generalisations of this result to more general classes of flows and nonlinearities.

\paragraph{More general periodic groups} In \cref{thm:main-result}, the dispersive extension comes in the form of a forced Schrödinger equation. In fact, the same localisation result holds for a more general class of periodic groups. Let $\tilde{S}_\tau$ be a periodic $C_0$ group on a Banach algebra $X$ with generator $A$ and period $\tau_P$, and let $M \coloneqq \sup_{\tau \in [0,\tau_P]} \|\tilde{S}_\tau\| < \infty$. Additionally, assume that for any real-analytic function $\rho : \C \to \C$ satisfying \eqref{eq:series-convergence} with convergence radius $R$, the map $u \mapsto \rho(u)$ is well-defined and analytic from $\{u \in X \,:\, \|u\|_X < \tfrac{R}{M}\}$ into $X$. Then, we have
\begin{equation*}
  \int_0^{\tau_P} \tilde{S}_{-\tau} \rho(\tilde{S}_\tau u) \dd\tau = w(\tau_P),
\end{equation*}
where $(v,w) \in C^0([0,\tau_P];X \times X)$ is the mild solution to
\begin{equation*}
  \begin{split}
    \partial_\tau v &= A v, \\
    \partial_\tau w &= A w + \rho(v)
  \end{split}
\end{equation*}
with $(v,w)(0) = (u,0)$. Consider the special case of $X = H^\ell_\per(0,L)$ and a general operator $A = -\ii\omega(\DD)$ with $\DD = -\ii \partial_x$ for a polynomial $\omega : \R \to \R$. Then, periodicity of $\tilde{S}_{\tau} = e^{-\ii\tau\omega(\DD)}$ with period $\tau_P$ is equivalent to $\tau_P \omega(\frac{2\pi K}{L}) \in 2\pi \Z$ for all $K \in \Z$. In this setting, the orbit average can also be expressed as a resonant nonlinearity $\tilde{\Ncal}_\rho$ defined as in \eqref{eq:resonant-nonlinearity} with $\Rcal_{a,b}$ replaced by
\begin{equation*}
  \tilde{\Rcal}_{a,b}(K) \coloneqq \Big\{(\j,\k) \in \Z^a \times \Z^b \,:\, \omega\Big(\frac{2\pi K}{L}\Big) = \sum_{m = 1}^a \omega\Big(\frac{2\pi\j_m}{L}\Big) - \sum_{m = 1}^b \omega\Big(\frac{2\pi\k_m}{L}\Big) \text{ and } K = \oneb_a \cdot \j - \oneb_b \cdot \k\Big\}
\end{equation*}
to encode the resonance condition enforced by the new dispersion relation $\omega$. Note that $\omega$ can be replaced by a generic smooth function. However, this does not yield an extended local system in general. Both statements are obtained by following the proofs in \cref{sec:localisation-result}.

\paragraph{A more general class of nonlinearities}
We now briefly discuss more general classes of nonlinearities. For simplicity, we restrict again to $S_\tau = e^{\ii\tau\partial_x^2}$ on $H^\ell_\per(0,L)$. First, we consider weighted orbital averages of the form
\begin{equation*}
  \int_0^{\tau_P} \phi(\tau) S_{-\tau} \rho(S_\tau u) \dd\tau
\end{equation*}
with $\phi \in L^1((0,\tau_P);\C)$, which appear in dispersion-managed NLS equations; see, for example, \cite{choi2023-06JMathAnalAppl}. These can also be expressed as the trace of the mild solution to a system of forced Schrödinger equations
\begin{equation*}
  \begin{split}
    \partial_\tau v &= \ii\partial_x^2 v, \\
    \partial_\tau w &= \ii \partial_x^2 w + \phi(\tau) \rho(v)
  \end{split}
\end{equation*}
with $(v,w)(0) = (u,0)$. The proof uses the same explicit Duhamel representation and additionally uses that $\tau \mapsto \phi(\tau) \rho(S_\tau u) \in L^1((0,\tau_P);H^\ell_\per(0,L))$ since $\rho(S_\tau u)$ is bounded. However, unless $\phi(\tau) \equiv \tau_P^{-1}$ (in which case we recover \cref{thm:main-result}), the weighted orbital average in general cannot be rewritten as a resonant nonlocal term similar to \cref{def:nonlinearity}.

Finally, we consider the case in which $\rho(u)$ is a polynomial function of $(u,\partial_x u,\dots,\partial_x^n u)$. In this case, we still obtain an analogue of \cref{thm:main-result}. However, since the Schrödinger group $S_\tau$ does not gain regularity, we obtain a mild solution of the extended forced Schrödinger system with $(v(\tau),w(\tau)) \in H^\ell_\per(0,L) \times H^{\ell-n}_\per(0,L)$, provided that $\ell > n + 1/2$ and using the continuity of $\rho : H^\ell_\per(0,L) \to H^{\ell-n}_\per(0,L)$. Therefore, an analogue of \cref{thm:main-result} can also be proved for a class of nonlinearities with additional internal Fourier multipliers. However, we do not make this precise here.

\section*{Acknowledgements}

B.H.~is funded by the Deutsche Forschungsgemeinschaft (DFG, German Research Foundation) -- Project-ID 543917644.

\paragraph{Use of generative AI} Generative AI tools (ChatGPT 5.6 Sol, Claude Opus 5) were used during the preparation of this manuscript, specifically for brainstorming ideas, additional literature review, and draft editing. All mathematical results and final formulations are the authors' own work.

\emergencystretch=3em
\printbibliography

\authordetails

\end{document}